\documentclass{birkjour}

\usepackage{subfiles}
\usepackage{hyperref}
\usepackage{cleveref}
\usepackage{amsmath}
\usepackage{multirow}
\usepackage{amssymb}
\usepackage{amsthm}
\usepackage{mathtools}
\usepackage{array}
\usepackage{thmtools}
\usepackage{thm-restate}
\usepackage{cleveref}
\usepackage{ytableau}
\usepackage{comment}
\usepackage{soul}
\usepackage{tikz}

\usetikzlibrary{positioning}
\usetikzlibrary{shapes.geometric}
\usetikzlibrary {shapes.misc}
\usetikzlibrary {arrows.meta}
\definecolor{brickred}{HTML}{B6321C}
\definecolor{royalblue}{HTML}{0071BC}

\usepackage{graphicx}
\usepackage{caption}
\usepackage{float}
\usepackage{subcaption}

\numberwithin{equation}{section}

\theoremstyle{definition}
\newtheorem{theorem}{Theorem}[section]
\newtheorem*{theorem*}{Theorem}

\newtheorem*{example*}{Example}
\newtheorem{lemma}[theorem]{Lemma}
\newtheorem*{lemma*}{Lemma}
\newtheorem{corollary}[theorem]{Corollary}
\newtheorem*{corollary*}{Corollary}

\newtheorem*{definition*}{Definition}

\newtheorem*{proposition*}{Proposition}

\newtheorem*{remark*}{Remark}

\begin{document}

\title[Two Conjectures on the Ungar Games]{Ungar Games on the Young-Fibonacci \\ and the Shifted Staircase Lattices}

\author[Y. Choi]{Yunseo Choi}
\address{Harvard University\\
Cambridge, MA, 02138}
\email{ychoi@college.harvard.edu}

\author[K. Gan]{Katelyn Gan}
\address{Sage Hill School\\
Newport Beach, CA, 92657}
\email{katelyngan77@gmail.com}

\begin{abstract}
In 2023, Defant and Li introduced the Ungar move, which sends an element $v$ of a finite meet-semilattice $L$ to the meet of some subset of the elements covered by $v$. More recently, Defant, Kravitz, and Williams introduced the Ungar game on $L$, in which two players take turns making Ungar moves starting from an element of $L$ until the player that cannot make a nontrivial Ungar move loses. In this note, we settle two conjectures by Defant, Kravitz, and Williams on the Ungar games on the Young-Fibonacci lattice and the lattices of the order ideals of shifted staircases.

\end{abstract}
\maketitle

\section{Introduction.} In 2023, Defant and Li \cite{Defant2023} introduced the \textit{Ungar move}, which sends an element $v$ of a meet-semilattice $L$ to the meet of $\{v\} \cup T$ for some subset $T$ of the elements that $v$ covers. If $T = \emptyset$, then the Ungar move is \textit{trivial}. In 2024, Defant, Kravitz, and Williams \cite{Defant2024} introduced the \textit{Ungar game} on a finite meet-semi lattice $L$, in which Atniss and Eeta alternate making nontrivial Ungar moves starting from  an element of $L$ until the player that cannot make a nontrivial Ungar move loses.

An element $v$ of a meet-semilattice $L$ is an \textit{Atniss win} if Atniss has a winning strategy in the Ungar game on the sublattice $[\hat{0}, v]$; otherwise, $v$ is an \textit{Eeta win} (see for example, \Cref{generallattice}). Let $\mathbf{A}(L)$ and $\mathbf{E}(L)$ be the set of Atniss and Eeta wins in $L$, respectively. For a graded lattice, let $\mathbf{A}_n(L)$ and $\mathbf{E}_n(L)$ be the set of Atniss and Eeta wins of rank $n$, respectively. 

\begin{figure}[htbp]
    \begin{center}
    \scalebox{1}{
    \begin{tikzpicture}[scale=0.8]
    \draw [color=brickred, line width=1pt ] (7.5, 11.25) circle (0.25cm) node{\normalsize A} ;
    \filldraw [ color=royalblue, fill=royalblue!15, line width=1pt ] (3.75, 12) circle (0.25cm) node {\normalsize E} ;
    \filldraw [ color=royalblue, fill=royalblue!15, line width=1pt ] (6.25, 10.5) circle (0.25cm) node {\normalsize E} ;
    \draw [ color=brickred,  line width=1pt ] (5,   11.25) circle (0.25cm) node {\normalsize A} ;
    \draw [ color=brickred,  line width=1pt ] (6.25, 12) circle (0.25cm) node {\normalsize A} ;
    \draw [ color=brickred,  line width=1pt ] (5,  12.75) circle (0.25cm) node {\normalsize A} ;
    \draw [ color=brickred,  line width=1pt ] (6.25, 13.75) circle (0.25cm) node {\normalsize A} ;
    \filldraw[black, thick] (5.2, 12.88) -- (6.05,13.6) ;
    \filldraw[black, thick] (3.99, 12.09) -- (4.8,12.6);
    \filldraw[black, thick] (6.25,13.5) -- (6.25,12.25);
    \filldraw[black, thick] (6.07,10.65) -- (5.21,  11.14);
    \filldraw[black, thick] (4.78,   11.38) -- (3.93, 11.84);
    \filldraw[black, thick] (5.2,   11.41) -- (6.05, 11.85);
    \filldraw[black, thick] (6.48, 10.58) -- (7.28, 11.13);
    \filldraw[black, thick] (7.27, 11.35) -- (6.47, 11.89);
    \filldraw[black, thick] (6.49, 13.65) -- (7.41, 11.47);
    \end{tikzpicture}}
    \end{center}    
    \caption{A lattice with Atniss wins unshaded and colored \textcolor{brickred}{red} and Eeta wins shaded and colored \textcolor{royalblue}{blue}.}\label{generallattice}
\end{figure}
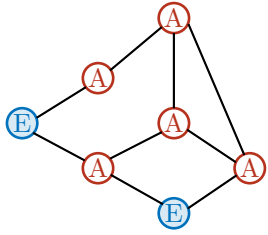

In \cite{Defant2024}, Defant, Kravitz, and Williams studied the Ungar games on the weak order on $S_n$, the intervals in Young's lattice, and the Tamari lattices. In this note, we settle two of Defant, Kravitz, and Williams' conjectures \cite{Defant2024} on the Ungar games on the Young-Fibonacci lattice and the lattices of the order ideals of shifted staircases. First, we characterize the Eeta wins of the Young-Fibonacci lattice $\mathbb{YF}$. In what follows, let $v_{i:j} = v_i v_{i+1} \cdots v_{j}$ for $v = v_{1} v_{2} \ldots v_{|v|}$.

\begin{theorem} \label{main1}
    For $r\geq 2$, an element $v \in \mathbb{YF}_r$ is an Eeta win if and only if
    \begin{itemize}
        \item $v_{1:|v|-1}=1 1 \cdots 1$ and the number of $1$s in $v$ is even, or
        \item $v_{1:|v|-1} \ne 1 1 \cdots 1$ and the number of $1$s to the left of the leftmost $2$ in $v$ is odd.
    \end{itemize}
\end{theorem}

Defant, Kravitz, and Williams' conjecture \cite[Conjecture 6.1]{Defant2024} follows. 

\begin{corollary} \label{corollary1}
    For $r \geq 2$, it holds that $|\textbf{E}_r(\mathbb{YF})|=f_{r-2}+(-1)^r$.
\end{corollary}

Next, we characterize the Eeta wins in the lattices $J(\mathrm{SS}_n)$ of the order ideals of shifted staircases, which corrects and settles \cite[Conjecture 6.2]{Defant2024}. 

\begin{theorem} \label{main2}
    An order ideal $v \in J(\mathrm{SS}_n)$ with binary representation $s$ is an Eeta win if and only if
    \begin{itemize}
        \item $s_{|s|}=0$, and 
        \item there are no odd-length sequences of $1$s followed by an odd-length sequence of $0$s in $s_{1:|s|-1}$.
    \end{itemize}
\end{theorem}

The sequence $|\mathbf{E}(J(\mathrm{SS}_n))|$ follows directly. 

\begin{corollary}\label{corollary2}
    The sequence $|\mathbf{E}(J(\mathrm{SS}_n))|$ is OEIS A061279 \cite{OEIS}.
\end{corollary}

The rest of this note is as follows. \Cref{prelims} is preliminaries. In Sections \ref{proof1} and \ref{proof2}, we prove Theorems \ref{main1} and \ref{main2}, respectively.

\section{Preliminaries.} \label{prelims}
A \textit{lattice} is a poset in which any two elements have a least upper bound (called their \textit{join} and notated by $\lor$) and a greatest lower bound (called their \textit{meet} and notated by $\land$). A \textit{meet-semilattice} is a poset in which every pair of elements has a meet. For $\{u, v\} \subseteq L$ such that $u \leq v$, the interval $[u, v]$ is all $w \in L$ such that $u \leq w \leq v$. If $|[u,v]|=2$, then say that $v$ \textit{covers} $u$ and write that $u \lessdot v$. The \textit{order ideal} of $v \in L$ is all $u \in L$ such that $u \leq v$. Let $\hat{0}$ be the minimal element of $L$. A \textit{graded lattice} $L$ has a rank function $\rho$ such that $\rho(\hat{0}) = 0$ and $\rho(u)+1 = \rho(v)$ if $u \lessdot v$. Let $\mathrm{Ung}(v)$ be the set of elements of $L$ that can be obtained by an Ungar move on $v$.

\section{Proof of \Cref{main1}.} \label{proof1}

The lattice $\mathbb{YF}$ was introduced by Fomin \cite{Fomin} and Stanley \cite{Stanley}. The elements of $\mathbb{YF}$ are the words of the alphabets $\{1,2\}$ (see for example, \Cref{fiblat} for $\mathbb{YF}$ up to rank $5$). The rank of $v \in \mathbb{YF}$ is $\rho(v) = \sum^{|v|}_{i=1} v_i$. Let $\mathbb{YF}_r$ be all elements in $\mathbb{YF}$ of rank $r$. Let $u \in \mathbb{YF}$. The lattice $\mathbb{YF}$ is such that $u \lessdot v$ if and only if

\begin{itemize}
    \item $v = u_{1:i}1u_{i+1:|u|}$ such that $u_{1:i}$ contains no $1$s, or
    \item $v = u_{1:i-1}2u_{i+1:|u|}$ such that $u_i$ is the leftmost $1$ in $u$,
\end{itemize}
where $u_{i:j} = \emptyset$ if $i > j$.

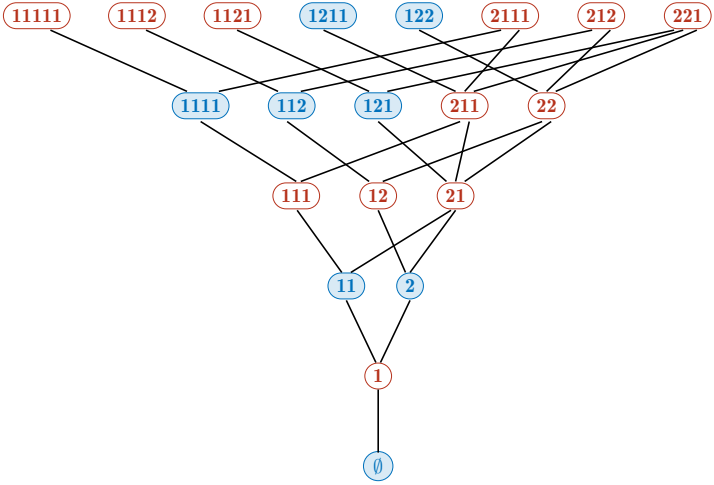
\begin{figure}[htbp]
    \centering    
        \begin{center}
    \scalebox{0.75}{
    \begin{tikzpicture}[scale=.8]

    \node[rounded rectangle, draw=royalblue, text=royalblue, fill=royalblue!15] at (4.5,0) {\textbf{$\emptyset$}};
    \node[rounded rectangle, draw=brickred, text=brickred] at (4.5,2) {\textbf{1}};
    \node[rounded rectangle, draw=royalblue, text=royalblue, fill=royalblue!15] at (3.8,4) {\textbf{11}};
    \node[rounded rectangle, draw=royalblue, text=royalblue, fill=royalblue!15] at (5.2,4) {\textbf{2}};
    \node[rounded rectangle, draw=brickred, text=brickred] at (2.8-0.1,6) {\textbf{111}};
    \node[rounded rectangle, draw=brickred, text=brickred] at (4.5,6) {\textbf{12}};
    \node[rounded rectangle, draw=brickred, text=brickred] at (6.2,6) {\textbf{21}};
    \node[rounded rectangle, draw=royalblue, text=royalblue, fill=royalblue!15] at (0.6,8) {\textbf{1111}};
    \node[rounded rectangle, draw=royalblue, text=royalblue, fill=royalblue!15] at (2.8-0.2,8) {\textbf{112}};
    \node[rounded rectangle, draw=royalblue, text=royalblue, fill=royalblue!15] at (4.5,8) {\textbf{121}};
    \node[rounded rectangle, draw=brickred, text=brickred] at (6.4,8) {\textbf{211}};
    \node[rounded rectangle, draw=brickred, text=brickred] at (8.2,8) {\textbf{22}};
    \node[rounded rectangle, draw=brickred, text=brickred] at (-3,10) {\textbf{11111}};
    \node[rounded rectangle, draw=brickred, text=brickred] at (-0.8,10) {\textbf{1112}};
    \node[rounded rectangle, draw=brickred, text=brickred] at (1.3,10) {\textbf{1121}};
    \node[rounded rectangle, draw=royalblue, text=royalblue, fill=royalblue!15] at (3.4,10) {\textbf{1211}};
    \node[rounded rectangle, draw=royalblue, text=royalblue, fill=royalblue!15] at (5.4,10) {\textbf{122}};
    \node[rounded rectangle, draw=brickred, text=brickred] at (7.4,10) {\textbf{2111}};
    \node[rounded rectangle, draw=brickred, text=brickred] at (9.4,10) {\textbf{212}};
    \node[rounded rectangle, draw=brickred, text=brickred] at (11.3,10) {\textbf{221}};
    
    \draw[black, thick] (4.5, 0.3) -- (4.5, 1.8-0.1);
    
    \draw[black, thick] (4.45, 2.2+0.1) -- (3.8, 3.68);
    \draw[black, thick] (4.55, 2.2+0.1) -- (5.2, 3.68);
    
    \draw[black, thick] (3.7, 4.3) -- (2.72,5.68);
    \draw[black, thick] (3.9,4.3) -- (6.1,5.68);
    \draw[black, thick] (5.1,4.3) -- (4.5,5.68);
    \draw[black, thick] (5.2,4.3) -- (6.2,5.68);
    
    \draw[black, thick] (2.7,6.32) -- (0.6,7.65);
    \draw[black, thick] (2.8,6.32) -- (6.3,7.65);
    \draw[black, thick] (4.3,6.32) -- (2.5,7.65);
    \draw[black, thick] (4.6,6.32) -- (8.2-0.1,7.65);
    \draw[black, thick] (6,6.32) -- (4.5,7.65);
    \draw[black, thick] (6.2,6.32) -- (6.5,7.65);
    \draw[black, thick] (6.4,6.32) -- (8.3,7.65);
    
    \draw[black, thick] (0.3,8.32) -- (-2.7,9.68);
    \draw[black, thick] (1.0,8.32) -- (7.2,9.68);
    \draw[black, thick] (2.3,8.32) -- (-0.6,9.68);
    \draw[black, thick] (2.8,8.32) -- (9.2,9.68);
    \draw[black, thick] (4.3,8.32) -- (1.4,9.68);
    \draw[black, thick] (4.7,8.32) -- (11,9.68);
    \draw[black, thick] (6.2,8.32) -- (3.3,9.68);
    \draw[black, thick] (6.4,8.32) -- (7.6,9.68);
    \draw[black, thick] (6.6,8.32) -- (11.2,9.68);
    \draw[black, thick] (8,8.32) -- (5.4,9.68);
    \draw[black, thick] (8.2,8.32) -- (9.6,9.68);
    \draw[black, thick] (8.4,8.32) -- (11.5,9.68);
    \end{tikzpicture}}
    \end{center} 
    \caption{The lattice $\mathbb{YF}$ up to rank $5$ with Atniss wins unshaded and colored \textcolor{brickred}{red} and Eeta wins shaded and colored \textcolor{royalblue}{blue}.}\label{fiblat}
\end{figure}

Now, we prove \Cref{main1}.

\begin{proof}[Proof of \Cref{main1}]
    We induct on the rank $r$. The statement holds when $r \in \{2, 3\}$. Now, suppose that $r \geq 4$ and that the statement holds for $r-1$ and $r-2$. First, suppose that $v \in \mathbb{YF}_{r}$ and that $v_{1} = 1$. Then $\mathrm{Ung}(v) = \{v_{2:|v|}\}$, because the only element that $v$ covers is $v_{2:|v|}$. Thus, $v \in \mathbf{E}_{r}(\mathbb{YF})$ if and only if $v_{2:|v|} \in \mathbf{A}_{r-1}(\mathbb{YF})$. Therefore, by the induction hypothesis, $v \in \mathbf{E}_{r}(\mathbb{YF})$ if and only if $v_{1:|v|-1}=1 1 \cdots 1$ and the number of $1$s in $v$ is even, or $v_{1:|v|-1} \ne 1 1 \cdots 1$ and the number of $1$s to the left of the leftmost $2$ in $v$ is odd.

    Next, suppose that $v \in \mathbb{YF}_{n}$ and that $v_{1} = 2$. Then $\{1v_{2:|v|}\} \subseteq \mathrm{Ung}(x)$, because $1 v_{2:|v|} \lessdot v$. Next, $\{v_{2:|v|}\} \subseteq \mathrm{Ung}(x)$, because if $v_2=1$, then $v_{2:|v|} = v_1v_{3:|v|} \wedge 1v_{2:|v|}$ and if $v_2=2$, then $v_{2:|v|} = v_1 1 v_{3:|v|} \wedge 1v_{2:|v|}$. By the induction hypothesis, either $v_{2:|v|}$ or $1v_{2:|v|}$ is an Eeta win. Thus, $v \in \mathbf{A}_{r}(\mathbb{YF})$. 
\end{proof}

\Cref{corollary1} now follows.

\begin{proof} [Proof of \Cref{corollary1}]We induct on the rank $r$. The statement holds for $r =2$. Now, suppose that $r\geq 3$ and that the statement holds for $r-1$. By \Cref{main1}, $v \in \mathbf{E}_{r}(\mathbb{YF})$ if and only if $v_1=1$ and $v_{2:|v|} \in \mathbf{A}_{r-1}(\mathbb{YF})$. Thus, $|\mathbf{E}_{r}(\mathbb{YF})| = f_{r-1} - |\mathbf{E}_{r-1}(\mathbb{YF})| = f_{r-2}+(-1)^{r}$ by the induction hypothesis.
\end{proof} 

\section{Proof of \Cref{main2}.} \label{proof2}

The $n^{\mathrm{th}}$ \textit{shifted staircase} $\mathrm{SS}_n$ consists of all lattice points $(i,j)$ such that $1\leq i \leq j \leq n$ (see for example, \Cref{ss5} for $\mathrm{SS}_{5}$). A natural bijection between the order ideals of $\mathrm{SS}_{n}$ and the length $n$ binary strings takes the path lying directly above an order ideal of $\mathrm{SS}_{n}$ and sends the up steps to $1$ and the down steps to $0$. The \textit{binary representation} of an order ideal of $\mathrm{SS}_n$ is the binary string that the order ideal corresponds to under the natural bijection. For a path in $\mathrm{SS}_{n}$ with binary representation $s$ such that $s_i=1$, say that $(\sum^{i}_{j=1}s_j, |s| - i + \sum^{i}_{j=1}s_j)$ is \textit{directly below} the $i^{\mathrm{th}}$ step of the path. For example, $(2,4)$ is directly below the third step of the path in \Cref{ss5}. 

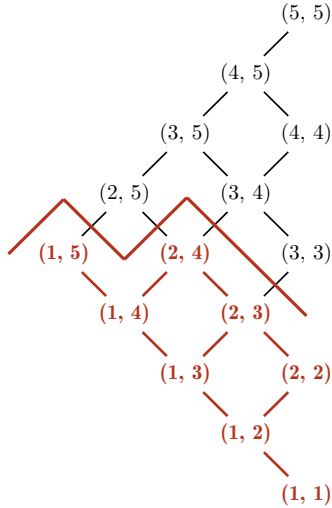
\begin{figure}[h]
    \centering
    \scalebox{0.8}{
    \begin{tikzpicture}[scale=1, baseline=20.5pt]
    \coordinate (A) at (0,0);
        \draw [shift={(A)}, color={brickred} , line width=1pt ] (4.5,0) node {\normalsize \textbf{(1, 1)}};
        \draw [shift={(A)}, color={brickred} , line width=1pt ] (4.5,2) node {\normalsize \textbf{(2, 2)}};
        \draw [shift={(A)}, color={black} , line width=1pt ] (4.5,4) node {\normalsize (3, 3)};
        \draw [shift={(A)}, color={black} , line width=1pt ] (4.5,6) node {\normalsize (4, 4)};
        \draw [shift={(A)}, color={black} , line width=1pt ] (4.5,8) node {\normalsize (5, 5)};
        \draw [shift={(A)}, color={brickred} , line width=1pt ] (3.5,1) node {\normalsize \textbf{(1, 2)}};
        \draw [shift={(A)}, color={brickred} , line width=1pt ] (3.5,3) node {\normalsize \textbf{(2, 3)}};
        \draw [shift={(A)}, color={black} , line width=1pt ] (3.5,5) node {\normalsize (3, 4)};
        \draw [shift={(A)}, color={black} , line width=1pt ] (3.5,7) node {\normalsize (4, 5)};
        \draw [shift={(A)}, color={brickred} , line width=1pt ] (2.5,2) node {\normalsize \textbf{(1, 3)}};
        \draw [shift={(A)}, color={brickred} , line width=1pt ] (2.5,4) node {\normalsize \textbf{(2, 4)}};
        \draw [shift={(A)}, color={black} , line width=1pt ] (2.5,6) node {\normalsize (3, 5)};    
        \draw [shift={(A)}, color={brickred} , line width=1pt ] (1.5,3) node {\normalsize \textbf{(1, 4)}};
        \draw [shift={(A)}, color={black} , line width=1pt ] (1.5,5) node {\normalsize (2, 5)};    
        \draw [shift={(A)}, color={brickred} , line width=1pt ] (0.5,4) node {\normalsize \textbf{(1, 5)}};
        
        \draw [shift={(A)}, brickred, line width=1.25 pt] (4.2, 0.3) -- (3.8, 0.7);
        \draw [shift={(A)}, brickred, line width=1.25 pt] (4.2, 2.3) -- (3.8, 2.7);
        \draw [shift={(A)}, black, thick] (4.2, 4.3) -- (3.8, 4.7);
        \draw [shift={(A)}, black, thick] (4.2, 6.3) -- (3.8, 6.7);
        \draw [shift={(A)}, brickred, line width=1.25 pt] (3.8, 1.3) -- (4.2, 1.7);
        \draw [shift={(A)}, black, thick] (3.8, 3.3) -- (4.2, 3.7);
        \draw [shift={(A)}, black, thick] (3.8, 5.3) -- (4.2, 5.7);
        \draw [shift={(A)}, black, thick] (3.8, 7.3) -- (4.2, 7.7);
        \draw [shift={(A)}, brickred, line width=1.25 pt] (3.2, 1.3) -- (2.8, 1.7);
        \draw [shift={(A)}, brickred, line width=1.25 pt] (3.2, 3.3) -- (2.8, 3.7);
        \draw [shift={(A)}, black, thick] (3.2, 5.3) -- (2.8, 5.7);
        \draw [shift={(A)}, brickred, line width=1.25 pt] (2.2, 2.3) -- (1.8, 2.7);
        \draw [shift={(A)}, black, thick] (2.2, 4.3) -- (1.8, 4.7);
        \draw [shift={(A)}, brickred, line width=1.25 pt] (1.2, 3.3) -- (0.8, 3.7);    
        \draw [shift={(A)}, brickred, line width=1.25 pt] (2.8, 2.3) -- (3.2, 2.7);
        \draw [shift={(A)}, black, thick] (2.8, 4.3) -- (3.2, 4.7);
        \draw [shift={(A)}, black, thick] (2.8, 6.3) -- (3.2, 6.7);    
        \draw [shift={(A)}, brickred, line width=1.25 pt] (1.8, 3.3) -- (2.2, 3.7);
        \draw [shift={(A)}, black, thick] (1.8, 5.3) -- (2.2, 5.7);    
        \draw [shift={(A)}, black, thick] (0.8, 4.3) -- (1.2, 4.7);    
        \draw [shift={(A)}, brickred, line width=1.5pt] (-0.41, 4) -- (0.5, 4.91); 
        \draw [shift={(A)}, brickred, line width=1.5pt] (0.5, 4.91) -- (1.5, 3.91);
        \draw [shift={(A)}, brickred, line width=1.5pt] (1.5, 3.91) -- (2.53, 4.94);
        \draw [shift={(A)}, brickred, line width=1.5pt] (2.53, 4.94) -- (4.50, 2.97);
        \end{tikzpicture}
        }
        \caption{The path above the order ideal of $SS_5$ with binary representation $10100$.} \label{ss5}
\end{figure}

For a binary string $s$, let $F(s)$ be the set of $i$ that satisfies
\begin{itemize}
    \item $s_i s_{i+1} = 10$ if $i\leq |s|-1$, or
    \item $s_i=1$ if $i=|s|$.
\end{itemize} 
For example, $F(110101)=\{2,4,6\}$. For $A \subseteq F(s)$, let $G(s, A)$ be obtained from $s$ by replacing $s_i$ with $0$ and $s_{i+1}$ (if exists) with $1$ for all $i \in A$. For example, $G(110101, \{2,6\})=101100$. If $|A| =1$, then we omit the brackets around $A$.

The elements of $J(\mathrm{SS}_{n})$ are the order ideals of $\mathrm{SS}_n$. The rank of $v \in J(\mathrm{SS}_n)$ with binary representation $s$ is $\rho(s) = 1 \cdot s_{|s|} + 2 \cdot s_{|s|-1} +\cdots +|s| \cdot s_1$. For example, the rank of an order ideal of $\mathrm{SS}_n$ with binary representation $1010$ is $6$. Let $v \in J(\mathrm{SS}_n)$ have binary representation $t$. The lattice $J(\mathrm{SS}_n)$ (see for example, \Cref{J(ss4)} for $J(\mathrm{SS}_4)$) is defined such that $u \lessdot v$ if and only if $G(t, i) = s$ for some $i \in F(t)$.  

Next, a \textit{$0$-block} (respectively, \textit{$1$-block}) in a string is a maximal substring of $0$'s (respectively, $1$'s). For example, in $110001$, there are two $1$-blocks and one $0$-block. Let $H(s)$ be the largest $i \in F(s)$ such that $s_i$ is the rightmost $1$ of an odd-length $1$-block that is followed by an odd-length $0$-block in $s$, provided such an $s_i$ exists. If such $s_i$ does not exist, then let $H(s) = \infty$, and call $s$ \textit{good}. Otherwise, $s$ is \textit{bad}. For example, $H(11010) = 4$, and $11010$ is bad. However, $H(110100) = \infty$, and $110100$ is good.

\begin{figure}
    \centering
    \scalebox{0.65}{
        \begin{tikzpicture}[scale=1]
            \node[rounded rectangle, draw=royalblue, text=royalblue, fill=royalblue!15] at (0,0) {0000};
            \node[rounded rectangle, draw=brickred, text=brickred] at (0,1) {0001};
            \node[rounded rectangle, draw=royalblue, text=royalblue, fill=royalblue!15] at (0,2) {0010};
            \node[rounded rectangle, draw=brickred, text=brickred] at (-2,3) {0100};
            \node[rounded rectangle, draw=brickred, text=brickred] at (2,3) {0011};
            \node[rounded rectangle, draw=brickred, text=brickred] at (0,4) {0101};
            \node[rounded rectangle, draw=royalblue, text=royalblue, fill=royalblue!15] at (-4,4) {1000};
            \node[rounded rectangle, draw=brickred, text=brickred] at (-2,5) {1001};
            \node[rounded rectangle, draw=royalblue, text=royalblue, fill=royalblue!15] at (2,5) {0110};
            \node[rounded rectangle, draw=brickred, text=brickred] at (0,6) {1010};
            \node[rounded rectangle, draw=brickred, text=brickred] at (4,6) {0111};
            \node[rounded rectangle, draw=brickred, text=brickred] at (2,7) {1011};
            \node[rounded rectangle, draw=royalblue, text=royalblue, fill=royalblue!15] at (-2,7) {1100};
            \node[rounded rectangle, draw=brickred, text=brickred] at (0,8) {1101};
            \node[rounded rectangle, draw=royalblue, text=royalblue, fill=royalblue!15] at (0,9) {1110};
            \node[rounded rectangle, draw=brickred, text=brickred] at (0,10) {1111};
            
            \draw[black, thick] (0, 0.25) -- (0, 0.75);
            \draw[black, thick] (0, 1.25) -- (0, 1.75);
            \draw[black, thick] (-0.4, 2.25) -- (-1.6, 2.75);
            \draw[black, thick] (0.4, 2.25) -- (1.6, 2.75);
            \draw[black, thick] (-2.4, 3.25) -- (-3.6, 3.75);
            \draw[black, thick] (-1.6, 3.25) -- (-0.4, 3.75);
            \draw[black, thick] (1.6, 3.25) -- (0.4, 3.75);
            \draw[black, thick] (-3.6, 4.25) -- (-2.4, 4.75);
            \draw[black, thick] (-0.4, 4.25) -- (-1.6, 4.75);
            \draw[black, thick] (0.4, 4.25) -- (1.6, 4.75);
            \draw[black, thick] (-1.6, 3.25+2) -- (-0.4, 3.75+2);
            \draw[black, thick] (-1.6, 3.25+4) -- (-0.4, 3.75+4);
            \draw[black, thick] (0.4, 2.25+4) -- (1.6, 2.75+4);
            \draw[black, thick] (-0.4, 2.25+4) -- (-1.6, 2.75+4);
            \draw[black, thick] (1.6, 3.25+2) -- (0.4, 3.75+2);
            \draw[black, thick] (1.6, 3.25+4) -- (0.4, 3.75+4);
            \draw[black, thick] (2.4, 5.25) -- (3.6, 5.75);
            \draw[black, thick] (3.6, 6.25) -- (2.4, 6.75);
            \draw[black, thick] (0, 8.25) -- (0, 8.75);
            \draw[black, thick] (0, 9.25) -- (0, 9.75);
        \end{tikzpicture}
        }
        \caption{The lattice $J(\mathrm{SS}_4)$ with Atniss wins unshaded and colored \textcolor{brickred}{red} and Eeta wins shaded and colored \textcolor{royalblue}{blue}.} \label{J(ss4)}
\end{figure}

We now prove auxiliary lemmas towards \Cref{main2}.

\begin{lemma} \label{lemma1}
    For any binary string $s$ and $T \subseteq F(s)$, the meet of all order ideals of $\mathrm{SS}_n$ with binary representations $G(s, i)$ for all $i \in T$ is the order ideal with binary representation  $G(s,T)$.
\end{lemma}

\begin{proof}
    Let $s$ be the binary representation of an order ideal $v \in J(\mathrm{SS}_n)$. For $i \in F(s)$, let $(x_i, y_i) \in v$ be the lattice point directly below the $i^{\text{th}}$ step of the path in $J(\mathrm{SS}_n)$ with binary representation $s$. In the bijection between the order ideals of $\mathrm{SS}_n$ and binary strings, $G(s,i)$ is the binary representation of $v \setminus \{(x_i, y_i)\}$. For $T \subseteq F(s)$, it holds that $$\wedge \{v \setminus \{(x_i,y_i)\} \mid i \in T\}= \cap \{v \setminus \{(x_i,y_i)\} \mid i \in T\},$$ which has binary representation $G(s,T)$.
    
\end{proof}

Now, we prove $G(s,I(s))$ is good for some $I(s) \subseteq F(s)$ for any binary string $s$.

\begin{lemma}\label{lemma2}
    For any binary string $s$, there exists some $I(s) \subseteq F(s)$ such that $G(s,I(s))$ is good. 
\end{lemma}

\begin{proof} 
    We recursively construct $I(s)$ such that 
    \begin{itemize}
        \item $|s| \not \in I(s)$, 
        \item $|s|-1 \not \in I(s)$ if $s_{|s|-1:|s|} = 00$, and
        \item $I(s) = I(s_{1:|s|-1})$ if $s_{|s|} = 1$. 
    \end{itemize}
    When $|s| \leq 2$, let $I(s) =  \{1\}$ if and only if $s=10$; otherwise, let $I(s):= \emptyset$. Now, $|s| \geq 3$. First, suppose that $s_{|s|}=1$. By construction, $|s|-1 \not \in I(s_{1:|s|-1})$. By the definition of good, $G(s_{1:|s|-1},I(s_{1:|s|-1}))1$ is good. Thus, $I(s) := I(s_{1:|s|-1})$ suffices.
    
    Next, suppose that $s_{|s|-1:|s|}=00$. By construction, $|s|-2 \not \in I(s_{1:|s|-2})$. By the definition of good, $G(s_{1:|s|-2},I(s_{1:|s|-2}))00$ is also good. Thus, $I(s) := I(s_{1:|s|-2})$ suffices.
        
    Now, suppose that $s_{|s|-1:|s|}=10$ and that $G(s_{1:|s|-1},I(s_{1:|s|-1}))0$ is good. By construction, $|s|-1 \not \in I(s_{1:|s|-1})$. Thus, $I(s) := I(s_{1:|s|-1})$ suffices.
    
    Next, suppose that $s_{|s|-2:|s|}=010$ and that $G(s_{1:|s|-1},I(s_{1:|s|-1}))0$ is bad. By construction, $|s|-2 \not \in I(s_{1:|s|-2}0)$. By construction and the definition of good, $G(s_{1:|s|-2}0, I(s_{1:|s|-2}0))1$ is good. Thus, $I(s) := I(s_{1:|s|-2}0) \cup \{|s|-1\}$ suffices.

    Lastly, suppose that $s_{|s|-2:|s|}=110$ and that $G(s_{1:|s|-1},I(s_{1:|s|-1}))0$ is bad. By construction, $I(s_{1:|s|-1}) = I(s_{1:|s|-2})$. Then $$G(s_{1:|s|-1}, I(s_{1:|s|-1}))_{1:|s|-2}=G(s_{1:|s|-2}, I(s_{1:|s|-2})).$$ Therefore, the rightmost $1$-block of $G(s_{1:|s|-2},I(s_{1:|s|-2}))$ has even length. By the definition of good, $G(s_{1:|s|-2}0, I(s_{1:|s|-2}))1$ is good. Thus, $I(s) := I(s_{1:|s|-2}) \cup \{|s|-1\}$ suffices.
    \end{proof}

We now prove \Cref{main2} using Lemmas \ref{lemma1} and \ref{lemma2}. That is, we prove that $v \in \mathbf{E}(J(\mathrm{SS}_n))$ if and only if the binary representation $s$ of $v$ satisfies $s_{|s|}=0$ and $s_{1:|s|-1}$ is good.

\begin{proof} [Proof of \Cref{main2}]
    Fix $n$. We induct on the rank $r$. The statement holds when $r=0$. Now, suppose that $r\geq 1$ and that the statement holds for $\leq r-1$. Let $v \in J(\mathrm{SS}_n)$ of rank $r$ with binary representation $s$. First, suppose that $s_{1:|s|-1}$ is good and that $s_{|s|}=1$. Let $u \in J(\mathrm{SS}_n)$ have binary representation $s_{1:|s|-1}0$. Since $u \lessdot v$ and $u \in \mathbf{E}(J(\mathrm{SS}_n))$ by the induction hypothesis, $v \in \mathbf{A}(J(\mathrm{SS}_n))$. 

    Next, suppose that $s_{1:|s|-1}$ is good and that $s_{|s|}=0$. Let $u \in \mathrm{Ung}(v)\setminus\{v\}$ have binary representation $t$. Then $t = G(s,T)$ for some nonempty $T \subseteq F(s)$ by \Cref{lemma1}. Now, if $|s|-1 \in T$, then $t_{|s|}=1$. Thus, $u \in \mathbf{A}(J(\mathrm{SS}_n))$ by the induction hypothesis. Next, suppose that $|s|-1 \not \in T$. Now, if the $0-$block containing $t_{\mathrm{max}(T)+1}$ has even length, then $H(t_{1:|s|-1}) = \mathrm{max}(T)+1$. If not, then the $1-$block containing $t_{\mathrm{max}(T)}$ in $t$ has even length, because $t_{1: |s|-1}$ is good, and so, $H(t_{1:|s|-1}) = \mathrm{max}(T)-1$. Thus, $t_{1:|s|-1}$ is bad, and $u \in \mathbf{A}(J(\mathrm{SS}_n))$ by the induction hypothesis. Therefore, $\mathrm{Ung}(v)\setminus\{v\} \subseteq \mathbf{A}(J(\mathrm{SS}_n))$, and $v \in \mathbf{E}(J(\mathrm{SS}_n))$.

    Finally, suppose that $s_{1:|s|-1}$ is bad. Then, $G(s_{1:|s|-1}, I(s_{1:|s|-1}))$ is good for some $I(s_{1:|s|-1})$ by \Cref{lemma2}. By \Cref{lemma1}, the meet of the order ideals of $\mathrm{SS}_n$ with binary representations $s_{1:|s|-1}0$ and $G(s, i)$ for all $i \in I(s_{1:|s|-1})$ has binary representation $t$ such that $t_{|s|}=0$ and $t_{1:|s|-1}$ is good. The order ideal with binary representation $t$ is in $\mathbf{E}(J(\mathrm{SS}_n))$ by the induction hypothesis. Thus, $v \in \mathbf{A}(J(\mathrm{SS}_n))$.
\end{proof}

\Cref{corollary2} immediately follows from \Cref{main2}.

\subsection*{Acknowledgment}
This project is a part of the PRIMES-USA program at the Department of Mathematics at MIT. The authors are grateful to Tanya Khovanova, Pavel Etingof, and Slava Gerovitch.

\end{document}